\documentclass[12pt]{amsart}
\usepackage{euscript,amsmath,amsfonts,amssymb,amsthm,comment}
\usepackage{hyperref}
\usepackage{graphicx}
\usepackage{amsaddr}

\makeatletter
\g@addto@macro{\endabstract}{\@setabstract}
\newcommand{\authorfootnotes}{\renewcommand\thefootnote{\@fnsymbol\c@footnote}}%
\makeatother

\numberwithin{equation}{section}
\usepackage{chngcntr}

\newtheorem{theorem}{Theorem}[section]

\counterwithout{theorem}{section}

\newtheorem{lemma}[theorem]{Lemma}

\newcommand{\ua}{\uparrow}
\newcommand{\da}{\downarrow}

\newcommand{\aside}[1]{\noindent\emph{#1}}

\begin{document}

\begin{center}
  \LARGE 
  Universality of single quantum gates \par \bigskip

  \normalsize
  \authorfootnotes
  Bela Bauer\textsuperscript{1},
  Claire Levaillant\textsuperscript{2},
  Michael Freedman\textsuperscript{1} \par \bigskip

  \textsuperscript{1}Station Q, Microsoft Research, Santa Barbara, CA 93106-6105, USA \par
  \textsuperscript{2}Department of Mathematics, University of California, Santa Barbara, CA 93106, USA\par \bigskip

  \today
\end{center}

\vspace{0.2in}

The quantum circuit model~\cite{Deutsch} envisions quantum information initialized in disjoint degrees of freedom, usually taken to be qubits (copies of $\mathbb{C}^2 = \operatorname{span}\{\ua,\da\}$)---the only case we consider. The information is then processed through a sequence of unitary ''gates'', each acting on a small number (usually $1$, $2$, or $3$) of tensor factors. Finally one or more qubits is read out by measuring in the $\sigma_z = \begin{array}{|rr|} 1 & 0 \\ 0 & -1 \end{array}$ basis. We treat the $2$-qubit {\sc Swap} gate
\begin{center}
    \begin{tabular}{c|cccc|}
        \multicolumn{1}{r}{} & $\ua\ua$ & $\ua\da$ & $\da\ua$ & \multicolumn{1}{r}{$\da\da$} \\
        $\ua\ua$ & $1$ & $0$ & $0$ & $0$ \\
        $\ua\da$ & $0$ & $0$ & $1$ & $0$ \\
        $\da\ua$ & $0$ & $1$ & $0$ & $0$ \\
        $\da\da$ & $0$ & $0$ & $0$ & $1$
    \end{tabular}
\end{center}
as inherent to the circuit model, i.e. the timelines of qubits can be permuted. Thus, for example, a given $2$-qubit gate can always be applied to any pair of qubits and in either order.

The proof of ``universality'' of a given set of gates, i.e. universality for the class BQP (polynomial time quantum computation), consists of two steps:
\begin{enumerate}
    \item showing that such a gate set is dense in $PU(2^n)$ for all $n$ ($=$ $\sharp$ of qubits in system), and
    \item checking polynomial efficiency, which is an exercise in the Kitaev-Solovay (K-T) algorithm~\cite{NC}.
\end{enumerate}
It is known \cite{BB} that if the single-qubit gates alone are dense in the projective unitary group $PU(2) = U(2)/U(1) \cong SO(3)$, then adding any additional $2$-qubit gate which is \emph{entangling} makes the gate set universal ($G$ ``entangling'' means there exists a vector $\phi\otimes\psi$ so that $G(\phi\otimes\psi)\neq \phi^\prime\otimes\psi^\prime$ for any $\phi^\prime$ and $\psi^\prime$).

Our theorem will also comprise these two aspects listed above but we will not comment on the efficiency aspect since this is by now routine and parallel to the discussion of Refs.~\cite{NC,DiVinc}.

\begin{theorem}
For some open dense set $\mathcal{O}\subset PU(4) = U(4)/U(1)$, of projective unitaries on $\mathbb{C}^2\otimes\mathbb{C}^2$, any gate $G\in\mathcal{O}$ is by itself universal for the class BQP (polynomial time quantum computation).
\end{theorem}

\aside{Remark 1}. Given \cite{BB} it is sufficient to prove that any element of $PU(4)$ can be approximated by some composition of $G$ and {\sc Swap} gates. For among these elements will be (up to phase) general transformations of the form $A\otimes\text{id}$, $A\in U(2)$. Thus the general single-qubit gate is a consequence of denseness in $PU(4)$. (Again the efficiency estimates are routine applications of the K-S algorithm and will not be given.)

\aside{Remark 2}. In the early days of the subject, the above theorem was stated by S.~Lloyd \cite{Lloyd}. After submission of this preprint, we were made aware of proofs of the statement in Refs.~\cite{Deutsch95,Weaver2000,Childs2010}, but decided to leave this manuscript available to add another perspective.

\begin{proof}
For some open dense $\mathcal{O}\subset PU(4)$, we will prove that for $G\in\mathcal{O}$, the set $\{G, \textsc{Swap}\circ G\circ\textsc{Swap}\}$ densely generates $PU(4)$; as remarked this is sufficient. First we exhibit a single $G$ with this property and then consider genericity.

Technically it is better to work with $su(4)$, the Lie algebra $\mathfrak g$ of $PU(4)$. We find by brute force an element $t\in \mathfrak g$ so that together with $\text{Ad}_\textsc{Swap}(t) := t^\prime$, these two elements generate $\mathfrak g$ as a Lie algebra. Explicitly if $t^i_j = t^{\alpha\gamma}_{\beta\delta}$ with $\alpha$, $\beta$, $\gamma$, $\delta$ qubit indices, then $t^{\prime\;\gamma\alpha}_{\delta\beta} = t^{\alpha\gamma}_{\beta\delta}$.

To avoid estimating round-off errors we used exact arithmetic, choosing $t_0$ essentially at random from traceless $4\times 4$ skew Hermitian matrices with entries of the form $a^i_j + \i b^i_j$, $1\leq i$, $j\leq 4$, $a$, $b$ small integers. Having chosen $t_0$ and computed $t_0^\prime$ we randomly applied Lie-brackets to produce new elements until some subset of $15$ ($= \operatorname{dim} su(4)$) of the matrices thereby produced became linearly independent. This was established by finding a \emph{non-zero determinant} when each of the $15$ matrices was itself regarded as a row vector of length $15$ within a $15\times 15$ matrix. (The $16^\text{th}$ entry is determined by the trace $=0$ condition and was therefore omitted.) Below we call this $15 \times 15$ determinant "$\mathrm{det}$".

Next we show that for some open dense set $\mathcal{Q} \subset su(4)$, $t\in\mathcal{Q}$ implies that $\{t,t^\prime\}$ generate $su(4)$ as a Lie algebra.
We checked Lie-generation of $su(4)$ by verifying an \emph{open} condition: $\text{det}\neq 0$.
The condition $\mathrm{det} = 0$ defines a (projective) real algebraic variety $\mathcal V$ inside $\mathbb{R}^{15}$ identified with $su(4)$ by
$$M_{ij} = (\sqrt{-1} M_{11},\ldots ,\sqrt{-1} M_{44}, \operatorname{Re} M_{12},\sqrt{-1} \operatorname{Im} M_{12},\ldots,\operatorname{Re} M_{34}, \sqrt{-1}\operatorname{Im} M_{34}).$$
\emph{Projective} means that the variety is a union of lines through the origin: $\vec v \in \mathcal V \implies a \vec{v} \in \mathcal V$.
A standard result~\cite{Hartshorne}, based on the implicit function theorem, states that varieties in $\mathbb{R}^n$ which are proper subsets of $\mathbb{R}^n$ are nowhere dense. Thus the points of $\mathbb{R}^{15}$ where $\mathrm{det} \neq 0$ form an open dense subset.
It is interesting that this use of algebraic geometry can be replaced by a short self-contained number-theoretic lemma, see the Appendix.

A fundamental link between the bracket of a Lie algebra and commutators in the group is the identity
$$
[s,t] = \lim_{\epsilon \rightarrow 0} \frac{1}{\epsilon^2} \left( e^{\epsilon s} e^{\epsilon t} e^{-\epsilon s} e^{-\epsilon t} \right)
$$
It follows that if some collection of brackets applied to $\lbrace t,t' \rbrace$ generate $su(4)$, then for $\epsilon$ sufficiently small, the elements $\lbrace e^{\epsilon t}, e^{\epsilon t'}\rbrace$ generate a dense subgroup of $SU(4)$. This is the key observation of the Kitaev-Solovay (K-S) algorithm~\cite{NC}. Unfortunately K-S does not give a uniform upper bound on the Killing norm $\Vert \epsilon t \Vert_K$ below which there is dense generation. This is because near the variety $\mathcal{V}$, $\epsilon$ will have to be smaller for the higher degree terms of the Campbell-Baker-Hausdorf (CBH) formula not to spoil the linear independence of commutations of the logarithms.

However the upper bound on $\Vert \epsilon t \Vert_K$ is a continuous function of the direction $t / \Vert t \Vert$, which we will refer to as $\epsilon_0(t/\Vert t \Vert_K)$. This function vanishes precisely along the projective directions of $\mathcal V$ where $\mathrm{det} = 0$. Thus there is an open subset of $SU(4)$:
$$
\mathcal{O} := \exp \lbrace t \in su(4) \setminus \mathcal{V} \  \big| \ \  \Vert t \Vert_K < \epsilon_0(t/\Vert t \Vert_K) \rbrace
$$
where K-S applies and $\lbrace e^{t}, e^{t'} \rbrace$
generate a dense subgroup of $SU(4)$. $\mathcal{O}$ is certainly \emph{not} dense, but fortunately there is a simple extension of K-S which removes the upper bound on $\Vert t \Vert$.

Consider the closed set $C$ of one-parameter subgroups of $SU(4)$ obtained as $\lbrace e^{x t} \rbrace$, where $\mathrm{det}(t) \neq 0$. Let $\bar{C}$ be the closed subset of $SU(4)$ which is the union over $C$. Since conjugation by {\sc Swap} commutes with taking powers, we may replace $g \in SU(4)$ with any power of $g$ in the search for a densely generating pair $\lbrace g,g' \rbrace := \lbrace g, {\sc Swap} \circ g \circ {\sc Swap} \rbrace$. Consider those $g$ belonging to the open dense subset $SU(4) \setminus C$ which obey the further condition that $g^k \in \mathcal{O}$ for some $k=1,2,3,\ldots$. Call this subset $\mathcal{U} \subset SU(4)$. Since $\mathcal{O}$ is open and group multiplication is continuous, $\mathcal{U}$ is a union of open subsets, thus $\mathcal{U}$ is an open subset. Now consider:

\begin{lemma}
Let $P \in G$ be a one-parameter subgroup of a compact Lie group $G$, with its induced topology. Let $I \subset P$ be any interval, then $\bigcup_{k \in \mathbb{Z}} I^k$ is dense in $P$.
\end{lemma}

\begin{proof}
The power $I^k$ is either all of $P$ or is itself an interval of $P$ of length $k \cdot \text{length}(I)$. The first case occurs if $\mathrm{id} \in I$ or $P$ is a circle subgroup. If $\mathrm{id} \notin I$ and $P$ is noncompact, then $\lbrace I^k \rbrace$ consist of non-nested intervals of increasing length on a curve of irrational slope on a $d$-torus. Using this model, the lemma reduces to the well-known fact that lines of irrational slope are dense in the d-torus.
\end{proof}

\begin{figure}
  \includegraphics{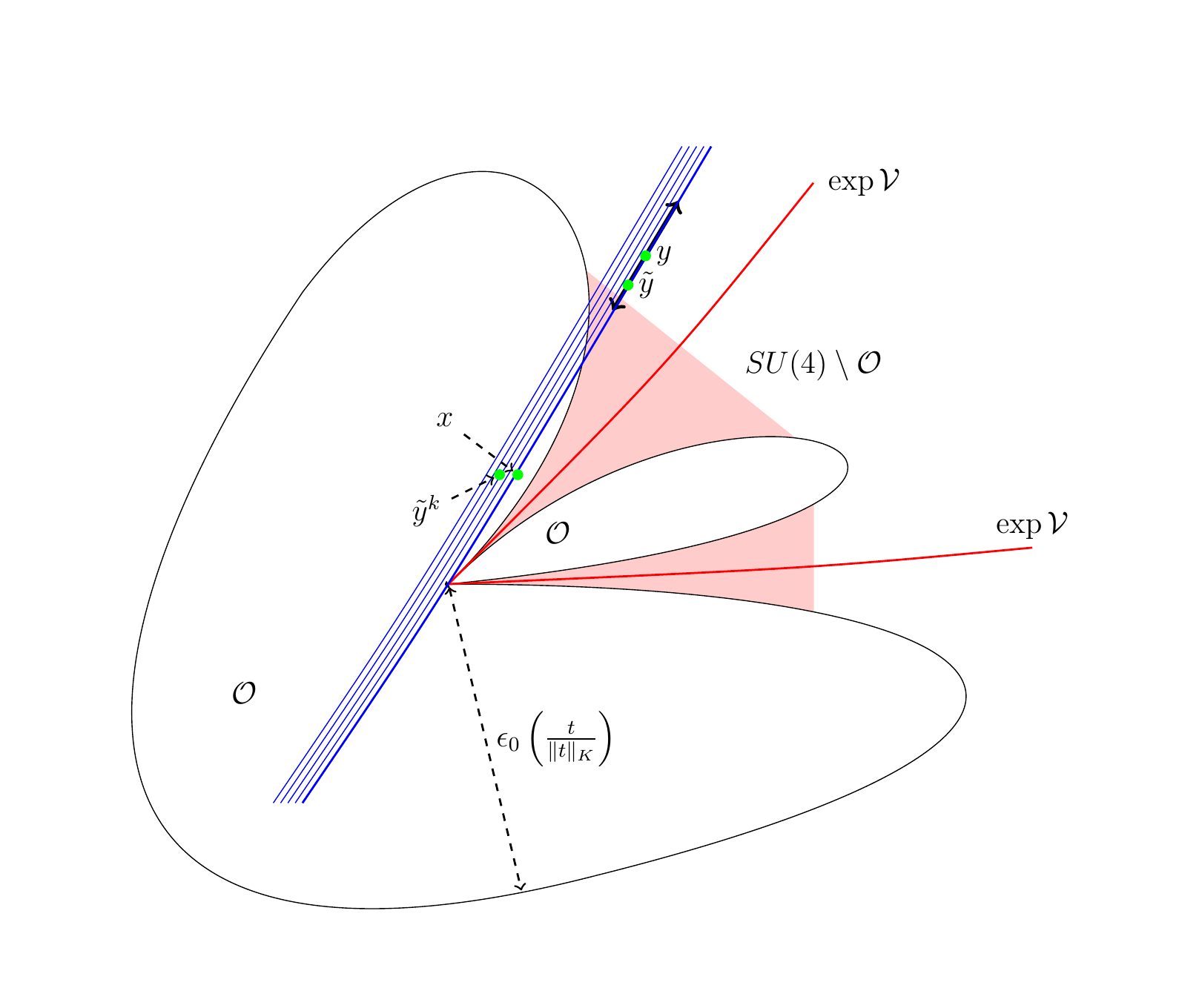}
  \caption{(Color online) Illustration for the proof of Theorem 1. Here, the parallel blue lines indicate $y$'s one-parameter subgroup, $P_y$. Shaded regions around $\exp \mathcal V$ indicate regions where $\Vert t \Vert_K \geq \epsilon_0(t/\Vert t \Vert_K)$. \label{fig} }
\end{figure}

Now consider the situation illustrated in Fig.~\ref{fig}.
For $y \in \mathcal{U}$, let $I$ be an arbitrarily small interval surrounding $y$ in its one-parameter subgroup. Using Lemma 2 and the openness of $\mathcal O$, we can obtain a $\tilde{y}$ such that $\tilde{y}^k \in \mathcal{O}$. Specifically, $\tilde{y}^k$ is required to approximate some small power $x=y^\delta$, $\delta \ll 1$, on the intersection of $y$'s 1-parameter subgroup with $\mathcal O$, $P_y \cup \mathcal O$. This establishes that $\mathcal U$ is dense in $SU(4) \setminus C$ and therefore dense in $SU(4)$, completing the proof of the theorem.
\end{proof}

\section{Appendix: Alternative proof}
\begin{lemma}
The complement of a proper algebraic variety over $\mathbb{R}$ (or $\mathbb{C}$) is dense (as well as open) in the usual topology on $V=\mathbb{R}^n$ (or $\mathbb{C}^n$).
\end{lemma}

\begin{proof}
First, a complex variety in $\mathbb{C}^n$ is also a real variety in $\mathbb{R}^{2n}$, so it is sufficient to consider the real case. To give an elementary proof, we first establish a lemma which is merely the simplest case of what is called the ``weak approximation theorem'' in the theory of adeles \cite{ant}. We thank Jeff Stopple and Keith Conrad for this reference.

Let $F$ be the number field $F = Q(\alpha) / \alpha^2 = 2$, and let $j_+(\alpha) = \sqrt{2}$ and $j_-(\alpha) = -\sqrt{2}$ define the two possible embeddings $j_\pm : F\longrightarrow\text{Reals} := R$. There is a Galois automorphism (involution) $g : F\longrightarrow F$, $g(\alpha) = -\alpha$ so that $j_+\circ g = j_-$ and $j_-\circ g = j_+$.

\begin{lemma}[Double Density Lemma]
The map $j: F\longrightarrow \mathbb{R} \times \mathbb{R}$ defined by $j(f) = (j_+(f), j_-(f))$ has a dense image with respect to the usual metric topology.
\end{lemma}

\begin{proof}
Fix any point $(p,q)\in \mathbb{R} \times \mathbb{R}$ and consider the equations:
$$x + \sqrt{2} y = p,$$
$$x - \sqrt{2} y = q.$$
Over $R$ they may be solved by $x = \frac{1}{2}(p + q)$ and $y = \frac{\sqrt{2}}{4}(p-q)$. Choosing rational approximations $x_0$ to $x$ and $y_0$ to $y$ we see that $j_+(x_0 + \sqrt{2} y_0)$ approximates $p$ and $j_-(x_0 + \sqrt{2} y_0)$ approximates $q$ to any desired precision.
\end{proof}

Of course both embeddings $j_+$ and $j_-$ are individually dense in $R$. Let $\mathcal{V} \subset V$ be a proper variety, i.e. $\mathcal{V} \neq V$, and let $t \in V \setminus \mathcal{V}$.
Consider a cubical neighborhood $\tau_{\epsilon}(t) = \lbrace s \ \big|\ \forall\ 0 < k \leq n: t_k - \epsilon < s < t_k + \epsilon \rbrace$. The double density lemma applied to each of the $n$ coordinates implies that for any $\epsilon > 0$ and any $t' \in V$, we can find an $s \in \tau_\epsilon(t')$ so that $g(s) \in \tau_\epsilon(t')$. Put another way, the Galois involution $g^{-1} = g$ scatters those elements of an $\epsilon$-neighborhood of $t_0$ with field $F$ coordinates uniformly over all of $V$, to form a dense set $S$ which must meet $\tau_\epsilon(t')$.

But the defining equation of $V$ is \emph{algebraic}, i.e. polynomial, so it holds equally before or after application of the field automorphism $g$. Consequently the sets $\mathcal V$ and $V \setminus \mathcal V$ are both preserved by $g$. Thus $V \setminus \mathcal V$ meets each $\tau_\epsilon (t')$, as above, and hence is dense.

\end{proof}

\section*{Acknowledgements}

We thank Adam Bouland for pointing out Ref.~\cite{Lloyd}, and L. Mancinska for pointing out Refs.~\cite{Deutsch95,Weaver2000,Childs2010}.

\end{document}